\documentclass{amsart}
\usepackage[utf8]{inputenc}
\usepackage{a4wide}
\usepackage{graphics}
\usepackage{color}
\usepackage{amssymb}
\usepackage[all]{xy}
\usepackage{amsmath}
\usepackage{stmaryrd}
\usepackage{subfigure}
\usepackage{longtable}
\usepackage{setspace}
\usepackage{booktabs}

\usepackage{pst-plot}
\usepackage{pstricks}
\usepackage{color}

\CompileMatrices

\newtheorem{theorem}{Theorem}[section]
\newtheorem{lemma}[theorem]{Lemma}

\newtheorem{corollary}[theorem]{Corollary}
\theoremstyle{definition}

\newtheorem{example}[theorem]{Example}

\theoremstyle{remark}

\newtheorem*{acknowledgement}{Acknowledgement}

\numberwithin{equation}{section}

\newcommand{\gq}{/\!\! /}
\newcommand{\cox}[1]{\mathcal{R}\left(#1\right)}

\newcommand{\Cl}[1]{\mathrm{Cl}\left(#1\right)}

\newcommand{\pic}[1]{\mathrm{Pic}\left(#1\right)}
\newcommand{\linpic}[2]{\mathrm{Pic}_{#1}\left(#2\right)}
\newcommand{\calO}[1]{\mathcal{O}\left(#1\right)}
\newcommand{\Spec}{\mathrm{Spec}}
\newcommand{\relspec}[2]{\mathrm{Spec}_{#1}\left(#2\right)}
\newcommand{\schnitt}[2]{\Gamma\left(#1,\mathcal{#2}\right)}
\newcommand{\schnittv}[2]{\Gamma\left(#1,#2\right)}

\newcommand{\WDiv}[1]{\mathrm{WDiv}\left(#1\right)}

\newcommand{\divi}[1]{\mathrm{div}\left(#1\right)}
\newcommand{\bbK}{\mathbb{K}}
\newcommand{\bbQ}{\mathbb{Q}}

\newcommand{\bbN}{\mathbb{Z}_{\ge 0}}
\newcommand{\ch}[1]{\mathbb{X}\left(#1\right)}

\newcounter{itemnumber}

\begin{document}
\title[Good quotients of Mori dream spaces]{Good quotients of Mori dream spaces}
\author[Hendrik B\"aker]{Hendrik B\"aker} \address{Mathematisches Institut, Universit\"at T\"ubingen, Auf der Morgenstelle 10, 72076 T\"ubingen, Germany}
\email{hendrik.baeker@mathematik.uni-tuebingen.de}

\begin{abstract}
We show that good quotients of algebraic varieties with finitely generated Cox ring have again finitely generated Cox ring.

\end{abstract}
\maketitle

\section{Introduction}

\noindent 
Let $X$ be a normal variety over some algebraically closed field $\bbK$ of characteristic zero.
If $X$ has finitely generated divisor class group and only constant invertible global functions then one can associate to $X$ a {\it Cox ring}; this is the graded $\bbK$-algebra
\[
 \cox{X}:=\bigoplus_{\Cl{X}}\schnittv{X}{\mathcal{O}_X(D)}.
\]
In the case of torsion in $\Cl{X}$ the precise definition requires a little care; see Section~2 for a reminder and \cite{coxscript} for details. We ask whether finite generation of the Cox ring is preserved when passing to the quotient by a group action. More precisely, for an action of a reductive group $G$ on $X$ we consider {\it good quotients}; by definition these are affine morphisms $\pi\colon U\rightarrow V$ with $\mathcal{O}_V=\left(\pi_*\mathcal{O}_U\right)^G$ where $U\subseteq X$ may be any open $G$-invariant subset.

\begin{theorem}
\label{thm2}
Let a reductive affine algebraic group $G$ act on a normal variety $X$ with finitely generated Cox ring $\cox{X}$, and let $U\subseteq X$ be an open invariant subset admitting a good quotient $\pi\colon U\rightarrow U\gq G$ such that $U\gq G$ has only constant invertible global functions. Then the Cox ring $\cox{U\gq G}$ is finitely generated as well.
\end{theorem}

Note that this statement was proven in \cite[Theorem~2.3]{HuKe} for the case that $X$ is affine with finite divisor class group and $U\gq G$ is a GIT-quotient. Moreover, in \cite[Remark~2.3.1]{HuKe} it was expected that GIT-quotients of Mori dreams spaces, i.e. $\bbQ$-factorial, projective varieties with finitely generated Cox ring, are again Mori dream spaces, which is a direct consequence of Theorem \ref{thm2}.

The following result is a step in the proof of Theorem \ref{thm2} but it also might be of independent interest. Let $K\subseteq \WDiv{X}$ be a finitely generated subgroup of Weil divisors. By the {\it sheaf of divisorial algebras} associated to $K$ we mean the sheaf of $\mathcal{O}_X$-algebras
\[
 \mathcal{S}:=\bigoplus_{D\in K}\mathcal{S}_D,\hspace{1cm}\mathcal{S}_D:=\mathcal{O}_X\left(D\right).
\]

\begin{theorem}
\label{thm1}
Let $X$ be a normal variety with finitely generated Cox ring
$\mathcal{R}(X)$.
Then, for any finitely generated subgroup
$K \subseteq \WDiv{X}$ and any
open subset $U \subseteq X$,
the algebra of sections
$\Gamma(U,\mathcal{S})$ of the sheaf
of divisorial algebras $\mathcal{S}$
associated to $K$ is finitely generated.
\end{theorem}

In particular, if $X$ has finitely generated Cox ring, then for every open subset $U\subseteq X$ the algebra of regular functions $\schnitt{U}{O}$ is finitely generated; note that even for affine varieties this fails in general, compare Example~\ref{ivanex}.

\section{Proof of Theorem 1.2}
Let us recall the construction of the Cox ring of a normal irreducible variety $X$ with finitely generated divisor class group and only constant invertible global functions. Fixing a finitely generated subgroup $K$ of the Weil divisors such that the projection $c\colon K\rightarrow \Cl{X}$ is surjective with kernel $K^0$, we can associate to $K$ the sheaf of divisorial $\mathcal{O}_X$-algebras $\mathcal{S}$. In order to identify the isomorphic homogeneous components of $\mathcal{S}$ we fix a character $\chi\colon K^0\rightarrow \bbK(X)^*$ such that $\divi{\chi(E)}=E$ holds for every $E\in K^0$ and consider the sheaf of ideals $\mathcal{I}$ locally generated by the sections $1-\chi(E)$ where $E$ runs through $K^0$ and $\chi(E)$ is homogeneous of degree $-E$.

The {\it Cox sheaf} is the sheaf $\mathcal{R}:=\mathcal{S}/\mathcal{I}$ together with the $\Cl{X}$-grading
\[
 \mathcal{R}=\bigoplus_{\left[D\right]\in\Cl{X}}\mathcal{R}_{\left[D\right]},\hspace{1cm} \mathcal{R}_{\left[D\right]}:=p\left(\bigoplus_{D'\in c^{-1}\left(\left[D\right]\right)}\mathcal{S}_{D'}\right),
\]
where $p:\mathcal{S}\rightarrow \mathcal{R}$ denotes the projection. The algebra of global sections is called the {\it Cox ring} of $X$, which is - up to isomorphy - independent of the choices of $K$ and $\chi$. For later use, note that by \cite[Lemma~3.3.5]{coxscript} for any open set $U\subseteq X$ we have
\[
 \schnitt{U}{R}\cong\schnitt{U}{S}/\schnitt{U}{I}.
\]
Moreover, from \cite[Lemma 4.2.2]{coxscript} we infer that the Cox ring is invariant when passing to a big open subset, i.e. an open subset whose complement is of codimension at least two. In particular, the two algebras $\schnitt{X^{\mathrm{reg}}}{R}$ and $\schnitt{X}{R}$ are equal, where $X^{\mathrm{reg}}$ denotes the set of regular points of $X$.

\begin{proof}[Proof of Theorem~\ref{thm1}]
For the first part of the proof we proceed similarly as in \cite[Proposition~5.1.4]{coxscript}. First assume that $K$ projects onto $\Cl{X}$. By $K^0$ we denote the subgroup of $K$ consisting of principal divisors, i.e., the kernel of the projection $c\colon K\rightarrow \Cl{X}$, and fix a basis $D_1,\ldots ,D_s$ for $K$, such that $K^0$ is generated by $a_1D_1,\ldots a_kD_k$ with certain $a_i\in\bbN$. Moreover, let $K^1\subseteq K$ be the subgroup generated by $D_{k+1},\ldots ,D_s$ and set $K':=K^0\oplus K^1$. We then have the associated Veronese subsheaves
\[
 \mathcal{S}^0:=\bigoplus_{D\in K^0}\mathcal{S}_D, \hspace{1cm}\mathcal{S}^1:=\bigoplus_{D\in K^1}\mathcal{S}_D, \hspace{1cm} \mathcal{S}':=\bigoplus_{D\in K'}\mathcal{S}_D.
\]
We claim that $\schnittv{X}{\mathcal{S}}$ is finitely generated. First note, that $\mathcal{S}_D\rightarrow \mathcal{R}_{\left[D\right]}$ is an isomorphism by \cite[Lemma 3.3.4]{coxscript}. Because $K^1\cong c\left(K^1\right)$ holds, these isomorphisms fit together to an isomorphism of sheaves
\[
 \mathcal{S}^1=\bigoplus_{D\in K^1}\mathcal{S}_D\rightarrow \bigoplus_{D\in c\left(K^1\right)} \mathcal{R}_{\left[D\right]}=:\mathcal{R}^1.
\]
Since $\schnitt{X}{R}$ is finitely generated the Veronese subalgebra $\schnittv{X}{\mathcal{R}^1}$ of the Cox ring is as well finitely generated (cf. \cite[Proposition 1.1.6]{coxscript}) which gives finite generation of $\schnittv{X}{\mathcal{S}^1}$. Every homogeneous function $f\in \schnittv{X}{\mathcal{S}'_{E_0+E_1}}$, where $E_i\in K^i$, is a product of a homogeneous section in $\schnittv{X}{\mathcal{S}_{E_1}}$ and an invertible section $g\in \schnittv{X}{\mathcal{S}_{E_0}}$, which itself is the product of certain $g_i^{\alpha_i}$ with $\divi{g_i}=a_iD_i$. Consequently, $\schnittv{X}{\mathcal{S}'}$ is generated by the functions $g_i$ and generators of $\schnittv{X}{\mathcal{S}^1}$; and thus is finitely generated. Since $K'$ is of finite index in $K$ the algebra $\schnittv{X}{\mathcal{S}}$ inherits finite generation from $\schnittv{X}{\mathcal{S}'}$ by \cite[Proposition~4.4]{K3}.

Now, let $U \subsetneq X$ be an arbitrary open subset. Then the complement $X\backslash U$ can be written as a union of the support of an effective divisor $D'$ and a closed subset of codimension at least two. Let $D\in K$ be a divisor which is linearly equivalent to $D'$, i.e. $D'=D+\divi{f}$ with a suitable rational function $f$. Then $f$ is contained in $\schnittv{X}{\mathcal{S}_D}$ and \cite[Remark~2.1.6]{coxscript} shows that $\schnitt{U}{S}=\schnitt{X}{S}_f$ is finitely generated.

Finally, if $K \subseteq \WDiv{X}$
does not project onto $\Cl{X}$,
then we take any finitely generated group
$\tilde{K} \subseteq \WDiv{X}$ with $K\subseteq \tilde{K}$ projecting onto
$\Cl{X}$ and obtain finite generation
of $\Gamma(U,\tilde{\mathcal{S}})$ for the associated
sheaf $\tilde{\mathcal{S}}$ of divisorial algebras.
This gives finite generation for the
Veronese subalgebra
$\Gamma(U,\mathcal{S}) \subseteq \Gamma(U,\tilde{\mathcal{S}})$
corresponding to $K \subseteq \tilde{K}$.
\end{proof}

\begin{corollary}
\label{cor}
 Let $X$ be normal variety with finitely generated Cox ring. Then for every open subset $U\subseteq X$ the algebra $\schnitt{U}{O}$ is finitely generated.
\end{corollary}
This observation allows us to construct normal affine varieties with non-finitely generated Cox ring.

\begin{example}
\label{ivanex}
Let $G$ be a connected semi-simple algebraic group and $H$ a unipotent subgroup such that the ring of invariants
 $$\schnitt{G}{O}^H=\schnitt{G/H}{O}$$
is not finitely generated. By \cite[Corollary~2.8]{grosshans} there is an open $G$-equivariant  embedding $G/H\subseteq X$ into a normal affine variety $X$. Since $G$ is semi-simple, $X$ has only constant invertible global functions and by the exact sequence
$$ 
\xymatrix{
0
\ar[r]
&
\pic{G/H}
\ar[r]
&
\pic{G}
}
$$
in  \cite[Proposition~3.2]{kraft} the divisor class group of $G/H$ is finitely generated. Consequently, $\Cl{X}$ is finitely generated as well but by Corollary~\ref{cor} the Cox ring  $\cox{X}$ is not finitely generated. For explicit examples see \cite[Section 2.4]{counterex}.
\end{example}

\section{Proof of theorem 1.1}
We consider a smooth irreducible algebraic variety $X$. Fix a finitely generated subgroup $K\subseteq\WDiv{X}$. By smoothness of $X$, the associated sheaf of divisorial algebras $\mathcal{S}$ is locally of finite type. This allows us to consider its relative spectrum over $X$ which we will denote by $\hat{X}:=\relspec{X}{\mathcal{S}}$. Note that the regular functions on $\hat{X}$ are precisely the global sections $\schnitt{X}{S}$. Since $\mathcal{S}$ is $K$-graded $\hat{X}$ comes with the action of the torus $H:=\Spec\, \bbK [K]$ and the canonical morphism $p\colon \hat{X}\rightarrow X$ is a good quotient for this action.

Now let an affine algebraic reductive group $G$ act on $X$. By a $G${\it -linearization} of the group $K$ we mean a lifting of the $G$-action to the relative spectrum $\hat{X}$ commuting with the $H$-action and making the projection $p$ equivariant. Any such $G$-linearization yields a $G$-representation on the regular functions of $\hat{X}$ via $g\cdot f(\hat{x})=f(g^{-1}\cdot\hat{x})$ and thereby induces a $G$-representation on $\schnitt{X}{S}$. In the special case where $K$ is a group of $G$-invariant divisors, \cite[Propositions~1.3 and~1.7]{gitweil} show that $K$ is canonically $G$-linearized and the induced representation on the global sections $\schnitt{X}{S}$ coincides with the action of $G$ on the rational functions of $X$ given by $g\cdot f(x)=f(g^{-1}\cdot x)$.

\begin{lemma}
\label{clfingen}
 Let an affine algebraic group $G$ act on the normal variety $X$ and let $U\subseteq X$ be an open $G$-invariant subset which admits a good quotient $\pi\colon U\rightarrow U\gq G$. If $\Cl{X}$ is finitely generated then $\Cl{U\gq G}$ is finitely generated as well.
\end{lemma}

\begin{proof}
 Without loss of generality we assume $X$ and $U\gq G$ to be smooth.
From \cite[Proposition 4.2]{kraft} we infer that the pullback homomorphism $\pi^*\colon \pic{U\gq G}\rightarrow \linpic{G}{U}$ into the classes of $G$-linearized line bundles is injective. It therefore suffices to show that $\linpic{G}{U}$ is finitely generated. By \cite[Lemma 2.2]{kraft} the following sequence is exact
$$ 
\xymatrix{
\mathrm{H}^1_{\mathrm{alg}}\left(G,\calO{U}^*\right)
\ar[r]
&
\linpic{G}{U}
\ar[r]
&
\pic{U}
}.
$$
Note that the group of algebraic cocycles $\mathrm{H}^1_{\mathrm{alg}}\left(G,\calO{U}^*\right)$ is finitely generated by the exact sequence in \cite[Proposition 2.3]{kraft}
$$ 
\xymatrix{
\ch{G}
\ar[r]
&
\mathrm{H}^1_{\mathrm{alg}}\left(G,\calO{U}^*\right)
\ar[r]
&
\mathrm{H}^1\left(G/G^0,E\left(U\right)\right)
},
$$
where $G/G^0$ is finite and $E\left(U\right)=\calO{U}^*/\bbK^*$ is finitely generated by \cite[Proposition~1.3]{kraft}.
\end{proof}

\begin{proof}[Proof of Theorem~\ref{thm2}]
Without loss of generality we assume $X$ and $U\gq G$ to be smooth. By Lemma~\ref{clfingen} we can choose a finitely generated group $K$ of Weil divisors on the quotient space $U\gq G$ projecting surjectively onto the divisor class group $\Cl{U\gq G}$. With $\mathcal{S}$ denoting the sheaf of divisorial algebras associated to $K$, the Cox ring $\cox{U\gq G}$ is the quotient of $\schnitt{U\gq G}{S}$ by the ideal $\schnitt{U\gq G}{I}$. Thus it suffices to show that the algebra of global sections $\schnitt{U\gq G}{S}$ is finitely generated. 

The pullback group $\pi^*K$ consists of invariant Weil divisors on $U$. It is therefore canonically $G$-linearized and we have the corresponding $G$-representation on the algebra $\schnitt{U}{T}$ where $\mathcal{T}$ denotes the sheaf of divisorial algebras associated to the group $\pi^*K$. We claim that we have a pullback homomorphism mapping $\schnitt{U\gq G}{S}$ injectively onto the algebra $\schnitt{U}{T}^G$ of invariant sections of $\schnitt{U}{T}$: 
\[
 \pi^*\colon \schnitt{U\gq G}{S}\ \rightarrow \ \schnitt{U}{T}^G,\qquad \schnittv{U\gq G}{\mathcal{S}_D}\ni f\ \mapsto \ \pi^*f\in\schnittv{U}{\mathcal{T}_{\pi^*D}}.
\]
We first note that every pullback section $\pi^{*}f\in \schnittv{U}{\mathcal{T}_{\pi^*D}}$ is indeed $G$-invariant because $\pi^*K$ is canonically $G$-linearized and $\pi^*f$ is $G$-invariant as a rational function on $U$. On each homogeneous component of $\schnitt{U\gq G}{S}$ the map $\pi^*$ is injective because it is the pullback with respect to the surjective morphism $\pi\colon U\rightarrow U\gq G$. Since $\pi^*$ is graded this yields injectivity of $\pi^*$ as an algebra homomorphism. For surjectivity it suffices to show that every homogeneous $G$-invariant section is a pullback section because the actions of $G$ and $H$ commute and, thus, $\schnitt{U}{T}^G$ is a graded subalgebra of $\schnitt{U}{T}$.  Consider a $G$-invariant homogeneous section $f\in\schnittv{U}{\mathcal{T}_{\pi^{*}D}}$. Since $f$ is invariant as a rational function in $\bbK(U)$ and it is regular on $U':=U\backslash \pi^{-1}\left(\mathrm{Supp}\left(D\right)\right)$ it descends to a regular function $\tilde{f}$ on $\pi(U')$ which is an open subset of $U\gq G$. Observe that we have
\[
\pi^*(\mathrm{div}(\tilde{f})+D) \ = \ \divi{f}+\pi^*D \ \ge \ 0.
\]
In particular, we obtain that the divisor $\mathrm{div}(\tilde{f})+D$ is effective and thus $\tilde{f}$ is a section in $\schnittv{U\gq G}{\mathcal{S}_D}$. By construction $f$ equals the pullback $\pi^*\tilde{f}$; hence our claim follows.

Thus the algebras $\schnitt{U\gq G}{S}$ and $\schnitt{U}{T}^G$ are isomorphic. The algebra $\schnitt{U}{T}$ is finitely generated by Theorem \ref{thm1}. Hilbert's Finiteness Theorem then shows that the invariant algebra $\schnitt{U}{T}^G$ is finitely generated as well.
\end{proof}

\begin{acknowledgement}
 I would like to thank J\"urgen Hausen and Ivan Arzhantsev for helpful suggestions and discussions.
\end{acknowledgement}

\end{document}